\documentclass[12pt,a4paper,oneside]{amsart}

\usepackage{amsmath,amsthm,amssymb}
\usepackage{fullpage}

\newtheorem{convention}{Convention}
\newtheorem{corollary}{Corollary}
\newtheorem{definition}{Definition}
\newtheorem{lemma}{Lemma}
\newtheorem{observation}{Observation}
\newtheorem{problem}{Problem}
\newtheorem{proposition}{Proposition}
\newtheorem{remark}{Remark}
\newtheorem{theorem}{Theorem}

\DeclareMathOperator{\conv}{conv}
\DeclareMathOperator{\spn}{span}

\DeclareMathOperator{\pos}{pos}
\DeclareMathOperator{\inte}{int}

\numberwithin{convention}{section}
\numberwithin{corollary}{section}
\numberwithin{definition}{section}
\numberwithin{lemma}{section}
\numberwithin{problem}{section}
\numberwithin{proposition}{section}
\numberwithin{remark}{section}
\numberwithin{theorem}{section}

\begin{document}
\title{A characterization of strongly monotypic polytopes}
\author{Vuong Bui}
\address{Vuong Bui, Institut f\"ur Informatik, Freie Universit\"{a}t
Berlin, Takustra{\ss}e~9, 14195 Berlin, Germany}
\thanks{The author is supported by the Deutsche Forschungsgemeinschaft
(DFG) Graduiertenkolleg ``Facets of Complexity'' (GRK 2434).  The work
was initiated during the time the author studied at Moscow Institute
of Physics and Technology.}
\email{bui.vuong@fu-berlin.de}

\begin{abstract}
	We characterize all the strongly monotypic polytopes.
	Hadwiger's conjecture for this class of polytopes is deduced
	from the characterization.
\end{abstract}

\maketitle

\section{Introduction}
Monotypic polytopes and strongly monotypic polytopes were first
introduced in \cite{mcmullen1974monotypic}.  Monotypic polytopes can
be seen as a subclass of simple polytopes: A polytope $P$ is monotypic
if every polytope having the same set of normals as $P$ is
combinatorially equivalent to $P$.

On the other hand, strongly monotypic polytopes is a subclass of
monotypic polytopes: A polytope $P$ is strongly monotypic if every
polytope $Q$ having the same set of normals as $P$ satisfies $\mathcal
A(P)$ and $\mathcal A(Q)$ being combinatorial equivalent, where
$\mathcal A(P)$ is the arrangement of hyperplanes containing the
facets of $P$.  An interesting property of such a polytope $P$ is
that: The intersection of any two translates of $P$ is always a
Minkowski summand of $P$.  In other words, a strongly monotypic
polytope is always a generating set.  The other direction
\emph{whether a polytope with the generating property is always a
strongly monotypic polytope} is still open.  However, it is shown to
be the case for $\mathbb R^3$ in \cite{mcmullen1974monotypic}.

Note that the property with a convex set in the place of a polytope
and any multiple number of sets instead of $2$ is actually the
definition of generating sets.  It is shown in
\cite{karasev2001characterization} that in order to check if a set is
generating, one just needs to check for every pair of its translates.
That is we can replace every pair of translates by any collection of
translates in the property of strongly monotypic polytopes.

The paper \cite{mcmullen1974monotypic} gives a partial
characterization of the strongly monotypic polytopes in $\mathbb R^3$,
which was later extended to a full characterization in
\cite{borowska2008strongly}.  In this paper, a treatment for higher
dimensions will be given.

Monotypic polytopes and strongly monotopic polytopes $P$ can be
recognized by the set of normals $N(P)$.  In
\cite{mcmullen1974monotypic}, several equivalent versions of the
necessary and sufficient conditions are given.  In particular,
Condition $M3'$ of a monotypic polytope $P$ says that: If $V_1$ and
$V_2$ are disjoint primitive subsets of $N(P)$ then $\pos V_1\cap\pos
V_2=\{0\}$.  ($V$ is a primitive subset of $N(P)$ if $V$ is linearly
independent and $\pos V\cap N(P)=V$.) Also, Condition $S4'$ for a
strongly monotypic polytope $P$ says that: If $Q$ is any polytope with
$N(Q)\subseteq N(P)$ then $Q$ is monotypic.

In order to characterize the set of normals of strongly monotypic
polytopes more conveniently, we give another equivalent condition.

\begin{theorem} \label{thm:monotypic-description}
	The following condition is necessary and sufficient for a
	polytope $P$ to be monotypic: If some $n+1$ normals of $P$ are
	in conical position, then their positive hull contains another
	normal of $P$.
\end{theorem}

In this text, a set of points is said to be separated from $0$ if
there is a hyperplane strictly separating the set from $0$ (i.e.  it
is not separated from $0$ if its convex hull contains $0$).  Also, a
set of points is said to be in \emph{conical position} if it is
separated from $0$ and none of its points is in the positive hull of
the others.

\begin{theorem} \label{thm:strongly-monotypic-description}
	An $n$-dimensional polytope $P$ is strongly monotypic if and
	only if every $n+1$ normals of $P$ are not in conical
	position.
\end{theorem}

The equivalences in Theorem \ref{thm:monotypic-description} and
Theorem \ref{thm:strongly-monotypic-description} are verified in
Section \ref{sec:equivalent-characterization} by using Conditions
$M3'$ and $S4'$.

Since two strongly monotypic polytopes of the same set of normals are
combinatorially equivalent, we characterize their sets of normals
instead of the polytopes themselves.  The problem of characterizing
the set of normals of strongly monotypic polytopes can be reduced to
the following problem, due to Theorem
\ref{thm:strongly-monotypic-description}.
\begin{problem} \label{prob:main-problem}
	Characterise all the finite sets $X\subset \mathbb
	R^n\setminus\{0\}$ such that (i) $0$ is in the interior of the
	convex hull of $X$, none $x\in X$ is a positive multiple of
	another $y\in X$ and (ii) every $n+1$ points of $X$ are not in
	conical position.
\end{problem}

The requirement (i) is actually to ensure that the set $X$ is a set of
normal vectors of a polytope.  From now on, all the sets are
understood as sets of normals, that is no element is a positive
multiple of another element.  They may be normalized under some
convention, for which the set of normals of a polytope is unique (e.g.
the norm of each normal is $1$, or Convention
\ref{cov:the-convention}, which we actually use and describe later).

For the sake of Problem \ref{prob:main-problem}, we give another name
for the definition of conical position (both names are used in the
text however.)
\begin{definition}
	A set of points is said to be in a \emph{good position} if it
	is \emph{not} in conical position.  Otherwise it is said to be
	in a \emph{bad position}.
\end{definition}
A point worth noting is that if a set of points is in a good position,
then any superset of that set is also in a good position.

% TODO: CHECK IF IT STILL WORKS FOR N<=2.
%
%Throughout the text we only treat the problem for $n\ge 3$, partly
%because the problem is trivial on the plane and partly because some
%statements are not true on the plane.  (Note that all polygons are
%strongly monotypic.)

Let $X$ be the set of normals of a strongly monotypic polytope.  We
begin the study of $X$ by showing that $X$ must contain a ``skeleton''
as defined in the following theorem.
\begin{theorem} \label{thm:skeleton}
	Every $X$ satisfying the condition in Problem
	\ref{prob:main-problem} contains some $k$ disjoint subsets
	$X_1,\dots, X_k$ such that (i) each $X_i$ is the set of
	vertices of a simplex whose relative interior contains $0$,
	(ii) the linear spaces spanned by each $X_i$ are linearly
	independent and (iii) these linear spaces directly sum up to
	$\mathbb R^n$.  Such a collection of $X_1,\dots,X_k$ is said
	to be a \emph{skeleton} of $X$.
\end{theorem}

We can assume that the linear spaces spanned by the sets $X_i$ are
orthogonal (otherwise applying an appropriate linear transform to
obtain it).  We also assume that the sum of the elements in each $X_i$
is the origin.  Theorem \ref{thm:skeleton} is proved in Section
\ref{sec:skeleton}.

Once we have the skeleton, we may have to add more points to obtain
$X$.  The following problem is the instance of Problem
\ref{prob:main-problem} for $k=1$.  Instances of higher $k$ can be
reduced to the instance of $k=1$ as in Section \ref{sec:reduction}.

% TODO: CHECK IF IT WORKS FOR N<=2
\begin{problem} \label{prob:k=1}
	Given the vertices $E=\{e_0,\dots,e_n\}$ of a simplex in
	$\mathbb R^n$ such that $e_0+\dots+e_n=0$ (i.e.  $0$ is in the
	interior), characterize all the sets $X\supseteq E$ such that
	every $n+1$ points of $X$ are in a good position.
\end{problem}

The solution for Problem \ref{prob:k=1} itself is given in Section
\ref{sec:k=1}.

Actually, the reason for the skeleton to be introduced so early in the
introduction is that it gives the following bound for the number of
translates in Hadwiger's conjecture \footnote{We are addressing the
combinatorial geometry one, not the graph theoretical one.}.

% TODO: CHECK IF IT STILL WORKS FOR N<=2. IT SEEMS STILL WORK, BUT WE NEED TO REARGUE IN SOME PLACES.
\begin{theorem} \label{thm:hadwiger-generating}
	If the set of normals $X$ of a strongly monotypic polytope
	$P\subset\mathbb R^n$ has a skeleton $X_1,\dots,X_k$, then $P$
	can be covered by at most $|X_1|\dots|X_k|$ translates of
	$(1-\epsilon) P$ for a small enough $\epsilon>0$.
\end{theorem}

The proof of Theorem \ref{thm:hadwiger-generating} given in Section
\ref{sec:hadwiger-generating} does not use all the characterization of
$X$ but only a handful of propositions.  It is perhaps not so easy to
deduce Theorem \ref{thm:hadwiger-generating} directly from Theorem
\ref{thm:strongly-monotypic-description} since the later only provides
a local picture of the normals while the former asks for a more global
picture.

Hadwiger's conjecture for this class of polytopes follows as
$|X_1|\dots|X_k|\le 2^n$.

The method of the proof in Section \ref{sec:hadwiger-generating}
relies on the property that we can tell which normals having their
facets intersecting at a vertex, for which we can employ the analysis
using the characterization. The method perhaps also works for other
classes of polytopes, the matter is that we have to deal with the
undeterministic nature of which normals having facets intersecting at
a vertex.

As for the state of Hadwiger's conjecture, it has been solved fully
only for convex bodies in $\mathbb R^2$ in \cite{levi1955uberdeckung}
by Levi.  Higher dimensions including $\mathbb R^3$ are still open
with only a few specific classes of convex bodies settled.

\section{Proof for the equivalence of the characterizations}
\label{sec:equivalent-characterization}
Let $D$ denote the condition in Theorem
\ref{thm:monotypic-description}.  In order to prove Theorem
\ref{thm:monotypic-description}, we will show that the two conditions
$M3'$ and $D$ are equivalent by verifying both directions.  Let us
fist remind Condition $M3'$ for a polytope $P$ to be monotypic, which
is actually already given in the introduction: If $V_1$ and $V_2$ are
disjoint primitive subsets of $N(P)$ then $\pos V_1\cap\pos
V_2=\{0\}$.

\begin{proposition} \label{prop:M3'->D}
	$M3'$ implies $D$.
\end{proposition}
\begin{proof}
	Suppose we do not have $D$, which means there exist $n+1$
	normals in conical position with the positive hull not
	containing any other normal.  Let $H$ be a hyperplane not
	through $0$ such that for each $a$ of the $n+1$ normals, ray
	$\overrightarrow{0a}$ cuts the hyperplane at only a point.
	Replacing every normal $a$ by the intersection of ray
	$\overrightarrow{0a}$ and the hyperplane, applying Radon's
	theorem to this set on the affine space, there will be two
	disjoint subsets whose intersection is nonempty.  It means the
	corresponding sets in the linear space contradict $M3'$.  The
	conclusion follows.
\end{proof}

To show that $D$ implies $M3'$, we need the following lemmas.

\begin{lemma} \label{lem:ray_as_intersection}
	Given two sets of normals $V_1, V_2$.  If $\pos V_1$ and $\pos
	V_2$ intersect at a point other than $0$, then there are
	$V'_1\subseteq V_1, V'_2\subseteq V_2$ such that $(\pos
	V'_1\cap\pos V'_2)\setminus\{0\}$ is precisely a ray in the
	relative interior of both the positive hulls.  Moreover, the
	union $V'_1\cup V'_2$ is separated from $0$.
\end{lemma}
\begin{proof}
	Let $U_1\subseteq V_1$ and $U_2\subseteq V_2$ be minimal
	subsets so that $\pos U_1\cap \pos U_2\ne \{0\}$.

	We show that $(\pos U_1\cap \pos U_2)\setminus\{0\}$ is a ray
	in the relative interior of both the positive hulls.  Suppose
	there are two rays $\overrightarrow{0p},\overrightarrow{0q}$
	in the intersection. 

	Let
	\[
		p=\sum_{x_i\in U_1} \lambda_i x_i=\sum_{x_i\in U_2}
		\lambda_i x_i,
	\]
	and
	\[
		q=\sum_{x_i\in U_1} \theta_i x_i=\sum_{x_i\in U_2}
		\theta_i x_i.
	\]
	(Note that there is no confusion of $\lambda_i$ for $x_i\in
	U_1$ and $x_i\in U_2$ since $U_1$ and $U_2$ are disjoint.  The
	same is for $\theta_i$.)

	Due to the minimality of $U_1, U_2$, all the coefficients
	$\lambda_i, \theta_i$ are positive, i.e.  the points $p,q$ are
	in the relative interior of $\pos U_1$ and $\pos U_2$.  (We
	can even conclude that the points in each of $U_1,U_2$ are
	linearly independent by Carath\'eodory's theorem for positive
	hulls.)

	Consider the point
	\[
		p-\alpha q = \sum_{x_i\in U_1}
		(\lambda_i-\alpha\theta_i) x_i = \sum_{x_i\in U_2}
		(\lambda_i-\alpha\theta_i) x_i,
	\]
	where $\alpha=\min\{ \min\{\lambda_i/\theta_i: x_i\in
	U_1\}, \min\{\lambda_i/\theta_i: x_i\in U_2\} \}$.
	The choice of $\alpha$ ensures that the coefficients
	$\lambda_i-\alpha\theta_i$ for $x_i\in U_1$ and
	$\lambda_i-\alpha\theta_i$ for $x_i\in U_2$ are all
	nonnegative with at least one zero.  Note that $p-\alpha q$ is
	not zero as $p,q$ are two different normals.  This contradicts
	with the minimality of $U_1, U_2$.

	We continue to prove that $U_1\cup U_2$ is separated from $0$.
	Suppose $0\in\conv(U_1\cup U_2)$, that is
	\[
		-\sum_{x_i\in U_1}\lambda_i x_i = \sum_{x_i\in
		U_2}\lambda_i x_i
	\]
	for some nonnegative coefficients $\lambda_i$.

	Since $(\pos U_1\cap\pos U_2)\setminus\{0\}$ is a ray in the
	relative interior of both positive hulls,
	\[
		\sum_{x_i\in U_1} \theta_i x_i = \sum_{x_i\in U_2}
		\theta_i x_i
	\]
	for some positive coefficients $\theta_i$.

	Let $\alpha=\max\{\lambda_i/\theta_i: x_i\in U_1\}$, the
	equation
	\[
		\sum_{x_i\in U_1} (-\lambda_i + \alpha\theta_i) x_i =
		\sum_{x_i\in U_2} (\lambda_i + \alpha\theta_i) x_i,
	\]
	has all nonnegative coefficients on both sides with at least a
	zero coefficient on the left.  It means that the positive hull
	of a proper subset of $U_1$ intersects $\pos U_2$ at a point
	other than $0$, contradicting the minimality of $U_1, U_2$.

	The sets $U_1,U_2$ confirm the conclusion.
\end{proof}

\begin{lemma} \label{lem:subspace_causes_contradiction}
	The followings are equivalent for a point set $X$ that spans
	$\mathbb R^n$:

	(i) There exist $n+1$ points of $X$ in conical position with
	the positive hull empty of other points of $X$.

	(ii) For some $d\le n$, there exist $d+1$ points of $X$ that
	span an $d$-dimensional space and in conical position with the
	positive hull empty of other points of $X$.
\end{lemma}
\begin{proof}
	The direction (i)$\implies$(ii) is trivial, just take some
	$(\dim \spn X) + 1$ points among them that span a $\dim\spn
	X$-dimensional space.

	We show the other direction, that is (ii)$\implies$(i).

	At first, we start with the given $d+1$ points $X'$ of
	dimension $d$, and increase their dimension one by one, by
	adding one point in each step, until we have $n+1$ points. If
	the current number of points is less than $n+1$, the dimension
	is therefore less than $n$, which implies the existence of
	another point $p$ not in the span of $X'$. If the positive
	hull of $X'\cup\{p\}$ is not empty of other points of $X$, we
	replace $p$ by any point of $X$ in $\pos(X'\cup\{p\})\setminus
	(X'\cup\{p\})$, and recursively repeat it until the positive
	hull is empty of other points of $X$. The process will
	eventually terminate as the positive hull contains fewer
	points after each step. So, for any $X'$, we can increase the
	dimension by one, by adding one point $p$ but still keeping
	the positive hull empty of other points in $X$. Therefore, we
	can come up with $n+1$ points satisfying the hypothesis, thus
	achieving (i).
\end{proof}
\begin{remark}
	If we have (ii), we also have the same (ii) for every higher
	$d$, including $n$, which is itself a stronger statement than
	(i).
\end{remark}

Now comes the verification of $D$ implies $M3'$.
\begin{proposition} \label{prop:D->M3'}
	$D$ implies $M3'$.
\end{proposition}
\begin{proof}
	Suppose we do not have $M3'$, which means there are two
	disjoint primitive subsets $V_1, V_2$ of $N(P)$ such that
	$(\pos V_1 \cap \pos V_2)\setminus\{0\}$ is nonempty.  Over
	all such pairs $V_1,V_2$, we consider a pair so that
	$\pos(V_1\cup V_2)$ is minimal (up to inclusion).  Due to
	Lemma \ref{lem:ray_as_intersection}, we have $(\pos
	V_1\cap\pos V_2)\setminus\{0\}$ is a precisely a ray in the
	relative interior of both $\pos V_1, \pos V_2$.  Moreover,
	$V_1\cup V_2$ is separated from $0$.

	Consider a hyperplane that cuts every ray
	$\overrightarrow{0p}$ for each $p\in V_1\cup V_2$ at exactly a
	point $p'$.  Let $U_1,U_2$ be the corresponding sets of those
	points $p'$.  Although $U_1,U_2$ are contained in an affine
	space, it is more convenient to see the hyperplane as a linear
	space by taking $\conv U_1\cap \conv U_2$ as the origin.
	Moreover, we assume the spaces spanned by $U_1$ and $U_2$ are
	orthogonal (otherwise, we transform the space).  Let $U$ be
	the intersection of all the normals in $\pos (V_1\cup V_2)$
	with the hyperplane.  Note that $\conv U_1$ and $\conv U_2$
	are empty of other points in $U$ (as $V_1,V_2$ are primitive).

	Suppose we have $D$, that is $U\setminus (U_1\cup U_2)$ is
	nonempty.

	Let $p$ be a point in $U\setminus (U_1\cup U_2)$ that has the
	smallest distance to $\conv U_1$.  It follows that
	$\conv(\{p\}\cup U_1)$ is empty of other points in $U$ since
	otherwise another point would be closer to $\conv U_1$ than
	$p$.  Also note that the points $\{p\}\cup U_1$ are in convex
	position since $p$ is not in the affine space spanned by
	$U_1$.

	Let $p'$ be the projection of $p$ onto the affine space of
	$U_1$.  Since $U_1$ is the vertices of a simplex containing
	$0$, there is a proper subset $U'_1\subset U_1$ so that
	$0\in\conv(\{p'\}\cup U'_1)$.  We then have $\conv(\{p\}\cup
	U'_1)$ and $\conv U_2$ intersecting.  Their corresponding
	normals are the two sets satisfying the condition of $V_1,V_2$
	but having a smaller positive hull of the union,
	contradiction.

	It follows that we do not have $D$.
\end{proof}

We now have verified Theorem \ref{thm:monotypic-description}.
\begin{proof}[Proof of Theorem \ref{thm:monotypic-description}]
	Theorem \ref{thm:monotypic-description} follows from
	Propositions \ref{prop:M3'->D} and \ref{prop:D->M3'}.
\end{proof}

Let $DD$ denote the condition in Theorem
\ref{thm:strongly-monotypic-description}, we show the equivalence of
the conditions $S4'$ and $DD$ in order to prove Theorem
\ref{thm:strongly-monotypic-description}.  Let us remind Condition
$S4'$ for a polytope $P$ to be strongly monotypic from the
introduction: If $Q$ is any polytope with $N(Q)\subseteq N(P)$ then
$Q$ is monotypic.

\begin{proof}[Proof of Theorem
\ref{thm:strongly-monotypic-description}]
	In one direction, if a monotypic polytope $P$ is also strongly
	monotypic, it should not have $n+1$ normals in conical
	position (i.e.  Condition $DD$).  Indeed, suppose $V$ is the
	set of such $n+1$ normals, we consider the subset of $N(P)$
	after removing every normal in the positive hull of $V$ except
	the normals in $V$ themselves, which is $N(P)\setminus((\pos
	V)\setminus V)$.  Any polytope taking this subset as the set
	of normals is not monotypic, due to the existence of the $n+1$
	normals of $V$ in conical position with the positive hull not
	containing any other normal.

	In the other direction, if a polytope $P$ has Condition $DD$,
	then every polytope $Q$ whose set of normals is a subset of
	the set of normals of $P$ should be monotypic.  It is rather
	clear since Condition $DD$ means the hypothesis in Condition
	$D$ does not happen.
\end{proof}

\section{The main technique in proofs}
\label{sec:main-technique}
Before proving the results in the following sections, we present the
technique that will be used throughout the text.  In each proof, in
order to show a bad position, we will show a set of $n+1$ points
$p_0,\dots,p_n$ so that some $n$ points of them are linearly
independent but all $n+1$ points are linearly dependent by a unique
(up to scaling) relation $\sum_i \lambda_i p_i=0$ with $2$ positive
coefficients and $2$ negative coefficients.  This set of points is
clearly in a bad position.

Usually, some $n$ points of them are shown to be linearly independent
by Lemma \ref{lem:base-for-linear-independence} as follows.

\begin{lemma} \label{lem:base-for-linear-independence}
	Given $E=\{e_0,\dots,e_n\}$ being the vertices of a simplex
	whose interior contains $0$.  For a point $p$ in the space
	spanned by $E$, the minimal subset of $E$ whose positive hull
	contains $p$ is called the \emph{support} of $p$.  Let $S_p$
	denote the support of $p$, we have every $p$ is not in the
	linear span of $E\setminus\{e_a,e_b\}$ for any $e_a\in S_p$
	and any $e_b\in E\setminus S_p$.
\end{lemma}
\begin{proof}
	Let $p=\sum_{e_i\in S_p} \lambda_i e_i$.

	Suppose $p$ is in the positive hull of
	$E\setminus\{e_a,e_b\}$, we also denote $p=\sum_{e_i\in
	E\setminus\{e_a,e_b\}}\theta_i e_i$.

	Taking the difference between the two representations,
	\[
		\Bigl(\sum_{e_i\in S_p\setminus\{e_a\}} (\lambda_i -
		\theta_i) e_i\Bigr) + \lambda_a e_a - \sum_{e_i\in
		(E\setminus S_p)\setminus\{e_b\}} \theta_i e_i = 0.
	\]

	It follows that $0$ is a linear combination of
	$|E\setminus\{e_b\}| = n$ points in $E$ with at least one
	nonzero coefficient $\lambda_a$.  This is a contradiction as
	the $n$ points are linearly independent.  The conclusion
	follows.
\end{proof}

Lemma \ref{lem:base-for-linear-independence} will be implicitly used
to show that some $n$ points are linearly independent in the way that:
One of them is not in the span of the remaining points, which are
themselves linearly independent.

\section{Proof of Theorem \ref{thm:skeleton}}
\label{sec:skeleton}
Among the elements of $X$ take a set of points $B$ such that $B$ is in
conical position, $B$ spans an $n$-dimensional space and its positive
hull is maximal in the sense that no other such points have the
positive hull being a superset of the positive hull of $B$.  Since
every $n+1$ points of $X$ are in a good position, it follows that $B$
is a set of $n$ linearly independent points $b_1,\dots,b_n$.

Consider any other point $x=\sum_{i=1}^n \lambda_i b_i$ in $X$.  The
$n+1$ points $\{x,b_1,\dots,b_n\}$ either (i) have the convex hull
containing $0$, for which all $\lambda_i$ are nonpositive or (ii) have
one of them in the positive hull of the others, for which all
$\lambda_i$ are nonnegative.  The other possibility that the
coefficients $\lambda_i$ are of mixed signs is impossible as follows.
First, we then have precisely one positive coefficient, since if there
are at least two positive coefficients, together with at least one
negative coefficients, the signs of coefficients indicates that
$\{x,b_1,\dots,b_n\}$ are in conical position, contradiction.  As
there are now only one positive coefficient, say $\lambda_i$, and
other nonpositive coefficients, we have
$b_i\in\pos\{x,b_1,\dots,b_n\}\setminus\{b_i\}$.  It is again a
contradiction since the $n$ points
$\{x,b_1,\dots,b_n\}\setminus\{b_i\}$ has a bigger positive hull than
the positive hull of $\{b_1,\dots,b_n\}$ (the former contains $x$
while the latter does not).  So, it is always either (i) or (ii).

Let the \emph{Cartesian support} \footnote{The leading ``Catersian''
is to distinguish with the other definition of support as in
Definition \ref{def:support}, which dominates the rest of the text.}
of a point $x$ be the set $\{b_i: \lambda_i\ne 0\}$ for
$x=\sum_{i=1}^n \lambda_i b_i$.

We show that the Cartesian support of two points both in
$\pos\{b_1,\dots,b_n\}$ or both in $\pos\{-b_1,\dots,-b_n\}$ are
either disjoint or one is a subset of the other.  Let the two points
be $x,y$ with $x=\sum_{i=1}^n \lambda_i b_i$ and $y=\sum_{i=1}^n
\theta_i b_i$.  Suppose the Cartesian support of $x$ be
$\{b_1,\dots,b_{k_2}\}$ and the Cartesian support of $y$ be
$\{b_{k_1},\dots,b_{k_3}\}$ with $1<k_1\le k_2<k_3$.  Since the points
in $\{y,b_1,\dots,b_n\}\setminus\{b_{k_1}\}$ are linearly independent,
the following linear relation
\[
	x=\Bigl(\sum_{i=1}^{k_1-1} \lambda_i b_i\Bigr) +
	\frac{\lambda_{k_1}}{\theta_{k_1}} y + \sum_{i=k_1+1}^{k_2}
	(\lambda_i - \frac{\lambda_{k_1}}{\theta_{k_1}} \theta_i) b_i
	- \sum_{i=k_2+1}^{k_3} \frac{\lambda_{k_1}}{\theta_{k_1}}
	\theta_i b_i
\]
is unique (upto scaling).  It raises a contradiction as the $n+1$
points $\{x,y\}\cup\{b_1,\dots,b_n\}\setminus\{b_{k_1}\}$ are in a bad
position with one positive coefficient of $x$ on the left hand side,
two positive and one negative coefficients of $b_1, y, b_{k_2+1}$ on
the right hand side.  (The readers can check Section
\ref{sec:main-technique} for the method.)

Since $0\in\inte\conv X$, there must be some $k$ points
$x_1,\dots,x_k$ all in $\pos\{-b_1,\dots,-b_n\}$ and their Cartesian
supports are disjoint while the union of the Cartesian supports is
$\{b_1,\dots,b_n\}$.  Let $X_i$ for each $i=1,\dots,k$ be the union of
the Cartesian support of $x_i$ and the point $x_i$ itself, we obtain
the desired sets $X_1,\dots,X_k$, which complete the proof.

\section{Reduction of Problem \ref{prob:main-problem} to Problem
\ref{prob:k=1}}
\label{sec:reduction}
Problem \ref{prob:main-problem} can be approached as follows.  We
first start with a skeleton $X_1,\dots,X_k$ of $X$ and then try to
constrain all the additional points in $X\setminus (X_1\cup\dots\cup
X_k)$.  The following propositions give the necessary condition for
the additional points.  Their proofs are rather not so hard and will
be given later in Section \ref{sec:proofs-of-propositions}.  Theorem
\ref{thm:reduction} reduces Problem \ref{prob:main-problem} to Problem
\ref{prob:k=1} when the remaining work is to characterize the points
in the spaces spanned by certain $X_i$.

We begin with a definition.
\begin{definition} \label{def:support}
	Given a skeleton $X_1,\dots,X_k$.  If $p_1,\dots,p_k$ is the
	projection of $p$ onto the linear spaces spanned by
	$X_1,\dots,X_k$ and $X'_i\subset X_i$ is the minimal subset of
	$X_i$ such that $p_i$ is in the positive hull of $X'_i$, then
	$S=X'_1\cup\dots\cup X'_k$ is called the \emph{support} of
	$p$.  We denote by $S_p$ the support of $p$.
\end{definition}
One can see that $|S_p|$ is at least $2$ except for $p\in
X_1\cup\dots\cup X_k$, and $|S_p|$ cannot exceed $n$.  Note that this
definition is compatible with the definition of support in Lemma
\ref{lem:base-for-linear-independence}, which is the instance for
$k=1$.

The following propositions are the necessary conditions for additional
points in $X$.

\begin{proposition} \label{prop:in-pos-of-2}
	If $x\in X$ is not in the space spanned by any $X_i$ then
	$x\in\pos\{x_i,x_j\}$ for some $x_i\in X_i, x_j\in X_j$.
\end{proposition}

\begin{proposition} \label{prop:disjoint-or-equal}
	If each of $x,y\in X$ is not in the space spanned by any $X_i$
	with $x\in\pos\{x_i,x_j\}$, $y\in\pos\{y_{i'}, y_{j'}\}$ for
	some $x_i\in X_i, x_j\in X_j$ and $y_{i'}\in X_{i'}, y_{j'}\in
	X_{j'}$, then the sets $\{i, j\}$ and $\{i', j'\}$ are either
	equal or disjoint.
\end{proposition}

\begin{proposition} \label{prop:same-region-same-point}
	Given two points $x,y\in X$ not in the space spanned by any
	$X_i$ with $x\in\pos\{x_i,x_j\}$, $y\in\pos\{y_i, y_j\}$ for
	some $x_i,y_i\in X_i$, $x_j,y_j\in X_j$.  If $|X_i|>2$ then
	$x_i=y_i$.
\end{proposition}

\begin{proposition} \label{prop:restrict-support}
	Suppose $x\in X$ is not in the space spanned by any $X_i$ and
	$x\in\pos\{x_i,x_j\}$ for some $x_i\in X_i, x_j\in X_j$.  If
	there is any other point $y\in X$ in the space spanned by
	$X_i$ than those points in $X_i$, then the support of $y$
	cannot contain $x_i$.
\end{proposition}

\begin{proposition} \label{prop:all-equal-n-positive}
	If $x\in X$ is not in the space spanned by any $X_i$,
	$x\in\pos\{x_i,x_j\}$ for some $x_i\in X_i, x_j\in X_j$, then
	every point $y\in X$ in the space spanned by $X_i$ is a
	positive multiple of $\sum_{e\in S_y} e$ (we will use the
	convention that $y=\sum_{e\in S_y} e$).  Further more, if
	$p,q$ are in the space spanned by $X_i$ then either their
	supports $S_p, S_q$ are disjoint, or one is a subset of the
	other.
\end{proposition}

An $X_i$ such that there is a point $x\in X$ contained in the relative
interior of $\pos\{x_i,x_j\}$ for some $x_i\in X_i, x_j\in X_j$ is
said to be an \emph{involved} $X_i$, otherwise it is called an
\emph{uninvolved} $X_i$.  If every $n+1$ points of $X$ in the space
spanned by any uninvolved $X_i$ are in a good position, then every
$n+1$ points in $X$ are always in a good position as in Theorem
\ref{thm:reduction} below.  That being said we have reduced Problem
\ref{prob:main-problem} to Problem \ref{prob:k=1}.  Note that our
settings for the points in the space spanned by an involved $X_i$ also
satisfy the solution of Problem \ref{prob:k=1} (as the readers can
compare with Section \ref{sec:k=1}).

\begin{theorem} \label{thm:reduction}
	If the requirements in all the above propositions are
	satisfied and the points in the space spanned by an uninvolved
	$X_i$ satisfy the condition in Problem \ref{prob:k=1}, then
	the points of $X$ satisfies the condition in Problem
	\ref{prob:main-problem}.
\end{theorem}
\begin{proof}
	Consider a set $C$ of any $n+1$ points of $X$.  From now on,
	we treat only the points in $C$ instead of those in $X$.  If
	the space of any uninvolved $X_i$ has more than $\dim X_i$
	points of $C$ then the points in $C$ are in a good position
	and we are done.  In the remaining case, there must be a pair
	$X_i,X_j$ such that there are more points of $C$ in the space
	spanned by $X_i\cup X_j$ than the dimension of $X_i\cup X_j$
	(by Proposition \ref{prop:disjoint-or-equal}).

	Let us first start with a lemma.
	\begin{lemma} \label{lem:presentation-by-many-points}
		Let $X_i$ be involved by a point in $\pos\{x_i,x_j\}$
		for some $x_i\in X_i$.  Consider some $|X_i|-1$ points
		in the space spanned by $X_i$, then either the points
		are in a good position, or $x_i$ can be found among
		the points, or $-x_i$ is the sum of some points among
		the considered points.  Moreover, if $|X_i|$ points
		are considered, then these points are always in a good
		position.
	\end{lemma}
	\begin{proof}
		Given a set of points in the space of $X_i$, we can
		establish the forest whose vertices are the supports
		of the points and a support is a father of another
		support if the former is superset of the latter (by
		Proposition \ref{prop:all-equal-n-positive}).

		Consider the two cases of the lemma:

		(i) $|X_i|-1$ points are considered.

		If $x_i$ is among the considered points, then the
		conclusion holds.

		We assume otherwise.  That is all the $\dim
		X_i=|X_i|-1$ points take only the subsets of
		$X_i\setminus\{x_i\}$ as their supports (by
		Proposition \ref{prop:restrict-support}).

		If the union of all the roots in the established
		forest is $X_i\setminus\{x_i\}$, then we have the
		desired representation of $-x_i$.  Indeed, let the
		points whose supports are the roots be $\{p_t\}_t$,
		then $\sum_{t} p_t= \sum_{e\in X_i\setminus\{x_i\}} e
		= -x_i$.  (Note that each $p_t$ is the sum of elements
		in the support of $p_t$, by Proposition
		\ref{prop:all-equal-n-positive}).

		If the union of all the roots is a proper subset of
		$X_i\setminus\{x_i\}$, then more points are considered
		than the cardinality of the union of the roots.  It
		follows that there is a vertex being the union of its
		children (easily by induction).  This indicates a good
		position since the point taking the father as the
		support is the sum of the points taking the children
		as the supports (also by Proposition
		\ref{prop:all-equal-n-positive}).

		(ii) $|X_i|$ points are considered.

		If the union of all the roots is $X_i$, then the sum
		of the points supported by the roots is $\sum_{e\in
		X_i} e = 0$, which indicates a good position.

		If the union of the roots is a proper subset of $X_i$,
		then there are more points considered than the
		cardinality of the union of the roots.  We also have a
		good position in this case (as already reasoned in
		(i)).
	\end{proof}

	We assume the dimension of $X_i\cup X_j$ is greater than $2$,
	otherwise, any three points from the plane are in a good
	position.  In case one of $X_i, X_j$ is $1$-dimensional, we
	always assume it is $X_j$.  We consider the following cases.

	\emph{Case 0}: If there is $0$ point of $C$ in $\spn (X_i\cup
	X_j)\setminus (\spn X_i\cup\spn X_j)$, then the space spanned
	by either $X_i$ or $X_j$ has more points of $C$ than its
	dimension, hence, they are in a good position.

	\emph{Case 1}: There is precisely $1$ point of $C$ in $\spn
	(X_i\cup X_j)\setminus (\spn X_i\cup \spn X_j)$, then the
	point is in $\pos\{x_i, x_j\}$ for some $x_i\in X_i, x_j\in
	X_j$.

	It suffices to treat the case when $\spn X_i$ has $\dim X_i$
	points of $C$ and $\spn X_j$ has $\dim X_j$ points of $C$, as
	otherwise one of them has more points than the dimension,
	which indicates a good position.

	Let the considered point be $y=\alpha x_i + \beta x_j$.

	By Lemma \ref{lem:presentation-by-many-points}, if we do not
	already have a good position in the space of $X_i$, then
	either $x_i\in C$ or $-x_i$ is the sum of some points in
	$C\cap\spn X_i$.  We have the same for $x_j$.  In either of
	the four combining cases, we always have a good position.
	Indeed, if $x_i\in C$ and $x_j\in C$, then $y,x_i,x_j$ are in
	a good position.  If $x_i\in C$ and $-x_j$ is the sum of the
	points $\{p_t\}_t$ in $C$, then $y+\beta \sum_{t} p_t = \alpha
	x_i$ implies that the points in $\{x_i,y\}\cup \{p_t\}_t$ are
	in a good position.  A similar argument also applies when the
	roles of $i,j$ are exchanged.  It remains to consider the case
	when both $-x_i,-x_j$ are the sums of the points in $C$, say
	$\{p_t\}_t$ and $\{q_s\}_s$, respectively.  In this case, the
	points in $\{y\}\cup \{p_t\}_t\cup \{q_s\}_s$ are in a good
	position since $y+\alpha\sum_t p_t +\beta\sum_s q_s = 0$.

	\emph{Case 2}: There are precisely $2$ points $y,z$ of $C$ in
	$\spn (X_i\cup X_j)\setminus (\spn X_i\cup \spn X_j)$.

	Let one point be in $\pos\{x_i, x_j\}$ for some $x_i\in X_i,
	x_j\in X_j$.  The other point must be in the same positive
	hull, or in $\pos\{x_i, -x_j\}$ in case $\spn X_j$ is
	$1$-dimensional (by Proposition
	\ref{prop:same-region-same-point}).  To cover all the cases,
	we say the two points are in $\pos\{x_i, x_j\}\cup
	\pos\{x_i,-x_j\}$.

	The space spanned by either $X_i$ or $X_j$ must have at least
	as many points as its dimension.

	(i) Suppose it is $X_i$.  By Lemma
	\ref{lem:presentation-by-many-points} either $x_i\in C$ or
	$-x_i$ is the sum of some points in $C\cap\spn X_i$.  No
	matter $x_i$ is in $\pos\{y,z\}$ or $y\in\pos\{x_i,z\}$ (or
	$z\in\pos\{x_i,y\}$), with the representation of $-x_i$ by the
	points in $C\cap\spn X_i$, we always have a good position by a
	similar argument to the one in Case 1.

	(ii) Suppose it is $X_j$, we are done if the space of $X_j$ is
	not $1$-dimensional, by the same argument as in (i).  It
	remains to consider the case of $1$-dimensional space, that is
	either $x_j$ or $-x_j$ is in $C$.  The three points $y,z,x_j$
	or $y,z,-x_j$ are always in a good position as they are on the
	same plane.

	\emph{Case 3}: If there are at least $3$ points of $C$ in
	$(\pos\{x_i, x_j\}\cup \pos\{x_i,-x_j\}) \setminus (\spn X_i
	\cup \spn X_j)$ for some $x_i\in X_i, x_j\in X_j$, then the
	three points, which are on the same plane, are always in a
	good position.

	All the cases were covered, and the conclusion follows.
\end{proof}

\section{Solution of Problem \ref{prob:k=1}}
\label{sec:k=1}
We first give all the necessary conditions for additional points of
$X\setminus E$ in the following propositions.  Their proofs are not so
hard and will be given in Section \ref{sec:proofs-of-propositions}.
Finally, Theorem \ref{thm:k=1} claims that these necessary conditions
are also sufficient.

In this section, the definition of support is still as in Lemma
\ref{lem:base-for-linear-independence} (which is also the support in
Definition \ref{def:support} for $k=1$).  We still denote by $S_p$ the
support of $p$.

As two vectors with one being a positive multiple of the other present
the same normal, we fist give a convention for convenience later.
\begin{convention} \label{cov:the-convention}
	Every point $p=\sum_{e_i\in S_p}\lambda_i e_i$ in $X$ is
	normalized so that $\min_{e_i\in S_p} \lambda_i = 1$.
\end{convention}

\begin{proposition} \label{prop:almost-homo}
	Every point $p=\sum_{e_i\in S_p}\lambda_i e_i$ in $X$ has all
	$\lambda_i=1$ except at most one $\lambda_j>1$.
\end{proposition}

\begin{definition}
Given $p=\sum_{e_i\in S_p}\lambda_i e_i$ in $X$, the point $e_j\in
S_p$ such that $\lambda_j>1$, if it exists, is said to be the
\emph{strange index} of $p$.  If there is such $p$, we also say $S_p$
has a strange index \footnote{Note that $S_p$ may be also the support
of another point, which may not have a strange index.  For
convenience, we use the point $e_i$ for the strange index instead of
the index $i$ as the name may suggest.}.
\end{definition}

\begin{proposition} \label{prop:intersection}
	Given two points $p, q$ in $X$, then either (i) $S_p\cap
	S_q=\emptyset$, or (ii) one of $S_p, S_q$ is a subset of the
	other, or (iii) $S_p\cup S_q=E$ and $|S_p\cap S_q|=1$, or (iv)
	$S_p\cup S_q=E$ and $|S_p\cap S_q|=n-1$.
\end{proposition}

\begin{proposition} \label{prop:same-support-same-strange-index}
	Given two points $p,q$ in $X$ having the same support
	$S_p=S_q$ of more than two points, then if both of them have
	strange indices, these indices are the same.
\end{proposition}

\begin{proposition}
\label{prop:disjoint-supports->at-most-1-has-strange-index}
	Given two points $p,q$ in $X$ with $S_p\cap S_q=\emptyset$
	then at most one of $p,q$ can have a strange index.
\end{proposition}

\begin{proposition} \label{prop:same-for-father-n-child}
	Given two points $p, q$ in $X$ with $S_p\subset S_q$,
	$|S_p|\ge 2$, $p=\sum_{e_i\in S_p} \lambda_i e_i$ and
	$q=\sum_{e_i\in S_q} \theta_i e_i$.  We then have
	$\lambda_i=\theta_i$ for every $e_i\in S_p$.
\end{proposition}

\begin{proposition}
\label{prop:support-insertecting-one-point->strange-index}
	Suppose $n\ge 3$.  Given two points $p,q$ in $X$ with $S_p\cup
	S_q=E$ and $S_p\cap S_q=\{e_k\}$.  If each of $S_p, S_q$ has
	at least $3$ points, then the strange index of any point of
	$p, q$, if any, is $e_k$.  If $|S_p|=2$ then the strange index
	of $S_q$, if any, is $e_k$.
\end{proposition}

\begin{proposition} \label{prop:2-strange-indices}
	Suppose $n\ge 3$.  Given two points $p,q$ in $X$ with $S_p\cup
	S_q=E$ and $|S_p\cap S_q|=n-1$ then $\lambda_a>1$,
	$\theta_b>1$ and $\lambda_a\theta_b \ge \lambda_a+\theta_b$
	where $\{e_a\}=S_p\setminus S_q$, $\{e_b\}=S_q\setminus S_p$
	and $p=\sum_{e_i\in S_p} \lambda_i e_i$, $q=\sum_{e_i\in S_q}
	\theta_i e_i$.
\end{proposition}

\begin{proposition}
\label{prop:two-big-supports-intersecting-at-one-point}
	Given two points $p,q$ in $X$ with $S_p\cup S_q=E$, $|S_p\cap
	S_q|=1$ and $|S_p|\ge 3, |S_q|\ge 3$, then for every point
	$r\in X$, the support $S_r$ is a subset of either $S_p$ or
	$S_q$.
\end{proposition}

\begin{proposition} \label{prop:big-intersection-constrain-others}
	Given two points $p,q$ in $X$ with $S_p\cup S_q=E, |S_p\cap
	S_q|=n-1$ and a point $r\in X$ with $S_r\ne S_p, S_r\ne S_q$
	and $|S_r|\ge 2$, then $S_r$ is either a subset of $S_p\cap
	S_q$ or $S_r=(S_p\setminus S_q)\cup (S_q\setminus S_p)$.
\end{proposition}

\begin{proposition} \label{prop:small-intersect-big}
	Given two points $p,q$ in $X$ with $|S_p\cap S_q|=1$,
	$|S_p|=2$, $|S_q|=n$ and a point $r\in X$ with $S_r\ne S_p,
	S_r\ne S_q$ and $|S_r|\ge 2$, then $S_r$ is either a subset of
	$S_q\setminus S_p$ or $S_r=(S_p\setminus S_q)\cup
	(S_q\setminus S_p)$.
\end{proposition}

We show that the above necessary requirements are actually sufficient.

\begin{theorem} \label{thm:k=1}
	If a set $X$ satisfies the conditions in all above
	propositions, then $X$ also satisfies the condition in
	Problem \ref{prob:k=1}.
\end{theorem}

\begin{proof}
	At first, we assme that $n\ge 3$ since it is known that for
	$n=1,2$, every segment/polygon is strongly monotypic.

	Take any $n+1$ points of $X$, consider the graph whose
	vertices are the supports of the points (each support presents
	only once, even if it supports multiple points).  There is a
	direct edge from $S_1$ to $S_2$ if and only if $S_1\subset
	S_2$.  This graph is actually a partial order, but for
	convenience, we remove some edges if necessary to obtain a
	forest in the way that: If a vertex $v$ is a subset of
	multiple vertices, then $v$ is connected to precisely one of
	them.

	We should first note that a vertex other than a root cannot
	support more than one point (by Proposition
	\ref{prop:same-for-father-n-child}).

	From now on in this proof, we only treat the $n+1$ considered
	points instead the whole $X$.

	\begin{lemma} \label{lem:lemma-for-k=1}
		If a tree supports \footnote{In case of confusion, the
		points supported by a tree are those supported by its
		vertices.} at least as many points as the cardinality
		of the root $S$ then either (i) one can obtain the sum
		of the elements in $S$ by a linear combination of the
		points with at most one negative coefficient (zero
		negative coefficients when $S$ has no strange index),
		or (ii) one of the points is in the positive hull of
		the others.  If the number of points supported by the
		tree is strictly greater than the number of elements
		of $S$, then we always have (ii).
	\end{lemma}
	\begin{proof}
		The statement is obviously true when $S$ contains only
		one point.

		Also, it is true when $|S|=2$.  Indeed, if there are
		two points $x,y$ in $\pos\{e_0,e_1\}$.  The desired
		sum $e_0+e_1$ is in the middle of the positive hull.
		Therefore, regardless of the position of $x,y$, the
		sum is always obtained by a linear combination of the
		two linearly independent vector $x,y$.  Note that the
		coefficients are not all negative, since otherwise,
		$e_0+e_1$ must be in the other quadrant
		$-\pos\{e_0,e_1\}$.  Also, when there is no strange
		index, $x,y$ are actually $e_0,e_1$, hence, there is
		no negative coefficient.  If there are $3$ points then
		one of them is in the positive hull of the other two
		points since one vector is in the angle formed by the
		other two vectors.

		We now prove the lemma for any $|S|\ge 3$ with the
		assumption that the lemma holds for $S$ of lower
		cardinalities.

		(a) Consider the case the tree supports exactly as
		many points as the cardinality of $S$.

		When no point has a strange index, we obtain (i) with
		the linear combination of only one summand which is
		the point supported by $S$.  We only consider the case
		we have a strange index hereafter.  Note that the
		strange index must be unique due to Propositions
		\ref{prop:same-for-father-n-child} and
		\ref{prop:same-support-same-strange-index}.  (Note
		that Proposition
		\ref{prop:same-support-same-strange-index} is the
		reason why we also consider $|S|=2$ in the base case
		of induction.)

		If there are $2$ points whose support is $S$, then the
		strange indices, if both have, are the same by
		Proposition
		\ref{prop:same-support-same-strange-index}.  Let the
		strange index be $s_k$, and the two points be
		$p=\sum_{e_i\in S} \lambda_i e_i$, $q=\sum_{e_i\in S}
		\theta_i e_i$ with $\lambda_k>\theta_k$, we then have
		\[
			\frac{\lambda_k-1}{\lambda_k-\theta_k} q -
			\frac{\theta_k-1}{\lambda_k-\theta_k}p =
			\sum_{e_i\in S} e_i.
		\]

		We now treat the case when only one point
		$p=\sum_{e_i\in S}\lambda_i e_i$ is supported by $S$.

		We have to consider only the case when no subtree
		supports more points than the cardinality of the root
		of that subtree, otherwise we have (ii) already by the
		induction hypothesis.  Such a situation implies that
		every subtree supports precisely as many points as the
		cardinality of the root of that subtree.  Also, the
		union of the children of a vertex $S'$ must be either
		$S'$ or $S'\setminus\{e_j\}$ for some $e_j\in S'$.  We
		only consider the latter case since in the former case
		the point supported by $S'$ is the sum of the points
		supported by the children.

		Let $e_j$ be the element of $S$ not in the children of
		$s$.  We consider the following cases.

		\emph{Case 1}: the point $e_j$ is the strange index.

		Let $S_1,\dots,S_m$ be the children of $S$, which
		support the points $q_1,\dots,q_m$.  Since no strange
		index is in any $S_t$, each $q_t$ is the sum of the
		elements in $S_t$.  It follows that
		\[
			(1-\frac{1}{\lambda_j}) \Bigl(\sum_{t=1}^m q_t\Bigr) +
			\frac{1}{\lambda_j} p = \sum_{e_i\in S} e_i.
		\]

		\emph{Case 2}: some $e_k$ other than $e_j$ is the
		strange index.

		Consider the smallest subtree whose root contains
		$e_k$.  If that subtree is rooted by $T$, then the
		union of the children, say $T_1,\dots, T_{m'}$, is
		$T\setminus\{e_k\}$.  (Note that we also cover the
		case the list of children is empty when $T=\{e_k\}$.)
		Let $q_1,\dots,q_{m'}$ be the points supported by
		these children, each $q_t$ is then the sum of the
		elements in $T_t$.  Let $q=\sum_{e_i\in T} \theta_i
		e_i$ be the point supported by $T$, we have
		\[
			p+\frac{\lambda_k-1}{\theta_k}\Bigl(\Bigl(\sum_{t=1}^{m'}
			q_t\Bigr) - q\Bigr) = \sum_{e_i\in S} e_i.
		\]

		(b) Consider the case the tree supports more points
		than the cardinality of $S$.

		If there are $3$ points whose support is $S$ then one
		of them is in the positive hull of the other two
		points (left as a small exercise for the readers).

		Consider the case when there are only two points $p,q$
		whose support is $S$.  The strange index $e_j$ cannot
		be in any of the children of $S$ (by Proposition
		\ref{prop:same-for-father-n-child}).  We assume that
		no subtree supports more points than the cardinality
		of its root, since otherwise we can apply the
		induction hypothesis.  Like the situation in (a), we
		also need to cover only the case the union of the
		children of $S$ is $S\setminus\{e_j\}$.  Let the point
		supported by the children be $x_1,\dots,x_m$.  Also,
		let $p=\sum_{e_i\in S} \lambda_i e_i$, $q=\sum_{e_i\in
		S} \theta_i e_i$ with $\lambda_j>\theta_j$, we then
		have
		\[
			q=\frac{\theta_j}{\lambda_j} p +
			(1-\frac{\theta_j}{\lambda_j}) \sum_{t=1}^m
			x_t.
		\]
		
		If there is only $1$ point supported by $S$, then we
		just need to consider the case one of the subtrees
		satisfies the requirement (b), as otherwise the union
		of the children of $S$ is $S$, which is trivial.  By
		applying induction hypothesis to the subtree, we have
		Conclusion (ii).
	\end{proof}
	
	With the fact that two disjoint supports cannot have two
	strange indices (by Proposition
	\ref{prop:disjoint-supports->at-most-1-has-strange-index}), we
	have the following corollary.
	\begin{corollary} \label{cor:the-corollary}
		If a forest with disjoint roots supports at least as
		many points as the cardinality of the union of the
		roots then either (i) one can obtain the sum of all
		the elements in the union of the roots by a linear
		combination of the points with at most one negative
		coefficient (zero negative coefficients when no root
		has a strange index), or (ii) one of the points is in
		the positive hull of the other.  If the number of
		supported points is strictly greater than the
		cardinality, then we always have (ii).
	\end{corollary}

	If a tree of the forest supports more points than the
	cardinality of the root then we are done, by Conclusion (ii)
	of Lemma \ref{lem:lemma-for-k=1}.

	Consider the remaining case that each tree supports at most as
	many points as the cardinality of the root.  The following
	cases are exclusive due to Propositions
	\ref{prop:intersection},
	\ref{prop:big-intersection-constrain-others} and
	\ref{prop:small-intersect-big}.

	\emph{Case 1}: The roots are disjoint (and the union is $E$).

	Applying Corollary \ref{cor:the-corollary} to the forest, we
	have a good position with either Conclusion (i) or Condition
	(ii).  (Note that the sum of the elements in the union $E$ is
	$0$.)
	
	\emph{Case 2}: There are precisely two roots $S_1, S_2$
	intersecting at only some $e_k$ and the union is $E$.

	Subcase (a): $|S_1|,|S_2| \ge 3$.

	If $e_k$ is among the $n+1$ points, then $e_k$ is in the
	positive hull of the two points $p=\sum_{e_i\in S_1} \lambda_i
	e_i$, $q=\sum_{e_i\in S_2} \theta_i e_i$ that are supported by
	$S_1, S_2$, since
	\[
		p+q-(\lambda_k+\theta_k-1)e_k = \sum_{e_i\in E} e_i = 0.
	\]
	(Note that the strange index of any point of $p,q$, if any, is
	$e_k$, by Proposition
	\ref{prop:support-insertecting-one-point->strange-index}.)

	Consider the case there is no $e_k$ among the $n+1$ points,
	then each support other than the two roots is either a subset
	of $S_1\setminus\{e_k\}$ or $S_2\setminus\{e_k\}$, by
	Propositions \ref{prop:intersection} and
	\ref{prop:two-big-supports-intersecting-at-one-point}.

	One of the two trees supports as many points as the
	cardinality of the root, since otherwise the total number of
	points is less than $n+1$.  Let it be the case for $S_1$.

	If there are two points supported by $S_1$, say
	$p=\sum_{e_i\in S_1} \lambda_i e_i, p'=\sum_{e_i\in S_1}
	\lambda'_i e_i$, let a point supported by $S_2$ be
	$\sum_{e_i\in S_2} \theta_i e_i$, the linear relation
	\[
		up + u'p' + q = \sum_{e_i\in E} e_i = 0,
	\]
	where $u,u'$ satisfy $u+u'=1$ and $\lambda_k u + \lambda'_k u'
	+ \theta_k = 1$ (with $\lambda_k\ne\lambda'_k$), shows that
	the three points $p,p',q$ are in a good position (any three
	points on the same plane through $0$ are in a good position).

	If there is only one point $p=\sum_{e_i\in S_1} \lambda_i e_i$
	supported by the root $S_1$, excluding $p$ and $S_1$ from its
	tree, there remains a forest with at least as many points as
	the cardinality of the union of the roots.  The union of the
	roots of the remaining forest is a subset of
	$S_1\setminus\{e_k\}$.  We are done if the union is a proper
	subset, since we have Conclusion (ii) of Corollary
	\ref{cor:the-corollary} then.  It remains to consider the case
	the union is precisely $S_1\setminus\{e_k\}$.  Also by
	applying Corollary \ref{cor:the-corollary} to the forest with
	no strange index, we have a linear combination $L$ of the
	remaining points with nonnegative coefficients whose value is
	the sum of the elements in $S_1\setminus\{e_k\}$.  The
	relation
	\[
		q-\frac{\theta_k-1}{\lambda_k} p +
		(1+\frac{\theta_k-1}{\lambda_k}) L = \sum_{e_i\in E}
		e_i = 0
	\]
	shows a good position of the points.
	
	Subcase (b): one of $|S_1|, |S_2|$ is $2$, say $|S_1|=2$.

	Suppose $S_1=\{e_0,e_1\}$, $S_2=\{e_1,\dots,e_n\}$.  By
	Proposition
	\ref{prop:support-insertecting-one-point->strange-index}, the
	strange index of $S_2$, if it exists, is $e_1$.  However, the
	strange index of $S_1$, if it exists, can be any of $e_0,e_1$.

	Consider the case $S_1$ supports two points $p=\sum_{e_i\in
	S_1} \lambda_i e_i$ and $p'=\sum_{e_i\in S_1} \lambda'_i e_i$,
	the linear relation
	\[
		up+u'p'+q=\sum_{e_i\in E} e_i = 0,
	\]
	where $u,u'$ satisfy $u\lambda_0+u'\lambda'_0=1,
	u\lambda_1+u'\lambda'_1+\theta_1=1$, shows a good position of
	$p,p',q$.

	Consider the case $S_2$ supports two points $q=\sum_{i=1}^n
	\theta_i e_i$ and $q'=\sum_{i=1}^n \theta'_i e_i$, the linear
	relation
	\[
		uq+u'q'+\frac{1}{\lambda_0}p=\sum_{e_i\in E} e_i = 0,
	\]
	where $u,u'$ satisfy $u+u'=1,
	u\theta_1+u'\theta'_1+\frac{\lambda_1}{\lambda_0}=1$, shows a
	good position of $q,q',p$.

	It remains to consider the case there are no two points both
	supported by either $S_1$ or $S_2$.

	If $e_0$ is among the considered points then the linear
	relation
	\[
		up+ve_0+q=\sum_{e_i\in E} e_i=0,
	\]
	where $u,v$ satisfy $u\lambda_0+v=1$ and
	$u\lambda_1+\theta_1=1$, shows a good position of $u,e_0,q$.

	If $e_1$ is among the considered points then the linear
	relation
	\[
		up+ve_1+q=\sum_{e_i\in E} e_i=0,
	\]
	where $u\lambda_0=1$ and $u\lambda_1 + v + \theta_1= 1$, shows
	a good position of $p,e_1,q$.

	If neither $e_0$ nor $e_1$ is among the points then all other
	points than $p,q$ are supported by subsets of
	$\{e_2,\dots,e_n\}$ (by Proposition
	\ref{prop:small-intersect-big}).  Applying Corollary
	\ref{cor:the-corollary} to these $n-1$ remaining points, one
	obtains a good position already (Conclusion (ii)) or otherwise
	a linear combination $L$ of nonnegative coefficients whose
	result is the sum of $e_2,\dots,e_n$.  The linear relation
	\[
		\frac{1}{\lambda_0}p +
		\frac{1-\frac{\lambda_1}{\lambda_0}}{\theta_1}q +
		\Bigl(1-\frac{1-\frac{\lambda_1}{\lambda_0}}{\theta_1}\Bigr)
		L = \sum_{e_i\in E} e_i = 0
	\]
	shows a good position of the considered points.

	\emph{Case 3}: There are precisely two roots $S_1, S_2$
	intersecting at precisely $n-1$ points and their union is $E$.
	Suppose $S_1=\{e_0,\dots,e_{n-1}\}$ and
	$S_2=\{e_1,\dots,e_n\}$.

	Consider the case there are two points $p=\sum_{i=0}^{n-1}
	\lambda_i e_i$ and $p'=\sum_{i=0}^{n-1} \lambda'_i e_i$
	supported by $S_1$.  Let the point supported by $S_2$ be
	$q=\sum_{i=1}^n \theta_i e_i$.  By Proposition
	\ref{prop:2-strange-indices}, we have
	$\lambda_i=\lambda'_i=\theta_i=1$ for $i=1,\dots,n-1$.  The
	linear relation
	\[
		up + u'p' + vq = \sum_{e_i\in E} e_i = 0,
	\]
	where $u,u',v$ satisfy $u\lambda_0 + u'\lambda'_0 = u+u'+v =
	v\theta_n=1$, shows that $p,p',q$ are in a good position.
	(Note that $\lambda_0\ne\lambda'_0$.)

	The case $S_2$ supports two points is treated likewise.

	It remains to check the case each $S_i$ supports only one
	point.

	As there is no other root, the support of every other point of
	the remaining $n-1$ points is a subset of
	$\{e_1,\dots,e_{n-1}\}$ (by Proposition
	\ref{prop:big-intersection-constrain-others}).  By Corollary
	\ref{cor:the-corollary}, we have either Conclusion (i) or
	Conclusion (ii).  We are done if it is Conclusion (ii).
	Otherwise, we can represent the sum of $e_1,\dots,e_{n-1}$ as
	a linear combination $L$ of the remaining $n-1$ points with
	nonnegative coefficient by Conclusion (i).  It follows that
	\[
		(\lambda_0\theta_n-\lambda_0-\theta_n) L + \theta_n p
		+ \lambda_0 q = \lambda_0\theta_n\sum_{e_i\in E} e_i =
		0,
	\]
	where the coefficient $\lambda_0\theta_n-\lambda_0-\theta_n$
	of $L$ is nonnegative by Proposition
	\ref{prop:2-strange-indices}.

	\emph{Case 4}: There are two roots $S_1,S_2$ as in Case 3 and
	additionally a root $S_0=\{e_0,e_n\}$.

	Let $r=\phi_0 e_0 + \phi_n e_n$ be the point supported by
	$S_0$, the relation
	\[
		up+vq+tr=\sum_{e_i\in E} e_i = 0,
	\]
	where $u,v,t$ satisfy $u\lambda_0+t\phi_0=1, u+v=1,
	v\theta_n+t\phi_n=1$, shows a good position of $p,q,r$.  (The
	readers can check that the determinant of the matrix
	corresponding to the system is $\phi_0\theta_n +
	\lambda_0\phi_n > 0$.)

	The four cases have covered all the situations, hence the
	conclusion follows.
\end{proof}

\section{Proof of Theorem \ref{thm:hadwiger-generating}}
\label{sec:hadwiger-generating}
The proof uses the characterization in Sections \ref{sec:reduction}
and \ref{sec:k=1}.  The readers just need to check the propositions
there.

At first, we revise an alternate formulation with illumination of
Hadwiger's conjecture whose exposition can be found in \cite[Chapter
VI]{boltyanski2012excursions}.

\begin{proposition}[Boltyanskii]
	A convex body $P$ can be covered by $h$ translates of
	$(1-\epsilon)P$ (for $\epsilon$ small enough) if and only if
	there exists a set $H$ of $h$ vectors such that for each point
	$x$ on the boundary of $P$, at least one vector $v\in H$ gives
	$x+v\in\inte P$.
\end{proposition}

Let $x$ be a point on the boundary of a polytope $P$.  The minimal
face containing $x$ is the intersection of some facets with the
hyperplanes $\{x: \langle a_i,x\rangle = c_i\}$.  The set of vectors
$v$ allows $x+\epsilon v\in\inte P$ (for some $\epsilon$ small enough)
is
\[
	\bigcap_i \{v: \langle a_i,v\rangle < 0\}.
\]
(Note that we use outward normals here.)

We can see that it is sufficient to find a set of vectors $H$ for the
vertices of $P$ only, instead of all the points on the boundary.

We give a trivial observation.
\begin{observation}
	Given a set of normals $\{a_i\}_i$ and $v\in \bigcap_i \{v:
	\langle a_i,v\rangle < 0\}$, if $a\in \pos\{a_i\}$ then
	$\langle a,v\rangle <0$.
\end{observation}

The above observation easily implies the following result.

\begin{lemma} \label{lem:inside-positive-hull}
	Given normals $a'_1,\dots,a'_m$ contained in the positive hull
	of normals $a_1,\dots,a_n$, then
	\[
		\bigcap_i \{v: \langle a'_i,v\rangle < 0\} \supseteq
		\bigcap_i \{v: \langle a_i,v\rangle < 0\}.
	\]
\end{lemma}

By the property of the skeleton, the union of the positive hulls of
$X'_1\cup\dots\cup X'_k$ over all $X'_i$ being a subset of $X_i$
obtained by excluding an element $x_i\in X_i$ is $\mathbb R^n$, that
is
\[
	\bigcup_{x_1\in X_1,\dots, x_k\in X_k} \pos (\bigcup_i
	X_i\setminus\{x_i\}) = \mathbb R^n.
\]

From each of $|X_1|\dots|X_k|$ such positive hulls of $n$ linearly
independent normals $a_i$, we pick any $v\in \bigcap_i \{v: \langle
a_i,v\rangle < 0\}$ to add to the desired set $H$.  Note that a vertex
is the intersection of the facets of $n$ primitive normals and no
other facets as otherwise there are $n+1$ normals in conical position
with the positive hull empty of other normals.  Using Lemma
\ref{lem:inside-positive-hull}, we can see that in order to finish the
proof, it suffices to show that every $n$ primitive normals are
contained in one of the $|X_1|\dots|X_k|$ positive hulls.

Assume otherwise, let $Y$ denote the set of the $n$ considered
normals, there is some $i$ so that the projection of $Y$ to the space
of $X_i$ is not contained in any positive hull of $|X_i|-1$ points of
$X_i$.  (From now on, we mean projection by the projection to the
space of $X_i$.) For convenience, for every two points $p,q$ in $Y$
with $S_{p'}\subset S_{q'}$ where $p',q'$ are the projections, we
discard $p$.  If the projections of multiple points are supported by
the same support, we keep only one of them in $Y$.  Also, we discard
every point whose projection is $0$.

\emph{Case 1}: The projections of every $2$ points of $Y$ have
disjoint supports.

If every point of $Y$ is in the space of $X_i$, then the sum of the
points in $Y$ is either $\sum_{e\in X_i} e=0$ (if there is no strange
index) or a positive multiple of the strange index (if it exists), by
Proposition \ref{prop:almost-homo}.

If there is some point $x\in Y$ not in the space of $X_i$, it follows
that $x\in\pos\{x_i,x_j\}$ for some $x_i\in X_i, x_j\in X_j$ (by
Proposition \ref{prop:in-pos-of-2}).

Consider the case that there is only one point $x$ in $Y\setminus\spn
X_i$.  The supports of the remaining points in $\spn X_i$ must be
subsets of $X_i\setminus\{x_i\}$ by Proposition
\ref{prop:restrict-support} (or actually also by the way of removing
points from $Y$ in the beginning) and they cover the whole
$X_i\setminus\{x_i\}$.  By Proposition
\ref{prop:all-equal-n-positive}, each point is the sum of the elements
in its supports.  Let $x=\alpha x_i + \beta x_j$, we have
\[
	\frac{1}{\alpha} x + \sum_{p\in Y\setminus\{x\}} p =
	\frac{\beta}{\alpha} x_j + \sum_{e\in X_i} e =
	\frac{\beta}{\alpha} x_j.
\]
This is a contradiction to the primitivity of $Y$ as its positive hull
contains $x_j$.

Consider the case there are two points in $Y\setminus\spn X_i$, say
the former $x$ and another point $y\in\pos\{y_i,y_j\}$ for some
$y_i\in X_i, y_j\in X_j$.  Since $x_i\ne y_i$ (by the way we remove
points from $Y$), we have $|X_i|=2$, by Proposition
\ref{prop:same-region-same-point}.  If $|X_j|=2$, then we have the
picture of the plane with four vectors $(0,1)$, $(0,-1)$, $(1,0)$,
$(-1,0)$ representing $X_i\cup X_j$.  As $x_i\ne y_i$, the two points
$x,y$ either have their convex hull containing $0$ or their positive
hull contains one of the four basis vectors.  If $|X_j|\ge 3$, we have
$x_j=y_j$ by Proposition \ref{prop:same-region-same-point}.  The
picture now is the the upper half-plane with $3$ vectors $(0,1)$,
$(-1,0)$, $(1,0)$ with $x_j=y_j=(0,1)$ and $X_i$ being the set of the
other two points.  As $x_i\ne y_i$, the positive hull of $x,y$
contains $(0,1)$.  In either case, we have a contradiction to the
primitivity of $Y$.

\emph{Case 2}: The supports of the projections of some $2$ points
$a,b$ of $Y$ are not disjoint.

As every point $x\in Y\setminus\spn X_i$ has the projection being the
same normal as some $x_i\in X_i$, the two points $a,b$ are actually in
the space of $X_i$.  Let $X_i=\{e_0,\dots,e_m\}$.  If $m=1$ or $m=2$,
then the two points $a,b$ are easily seen from the segment or the
plane to be not primitive.  We suppose $m\ge 3$.  It follows from
Proposition \ref{prop:intersection} that we need to consider the
following two situations.

Suppose $|S_a|=|S_b|=|X_i|-1$, say $S_a=\{e_0,\dots,e_{m-1}\},
S_b=\{e_1,\dots,e_m\}, a=\sum_{t=0}^{m-1}\lambda_t e_t,
b=\sum_{t=1}^m\theta_t e_t$, with
$\lambda_0>\lambda_1=\dots=\lambda_{m-1}=1$,
$1=\theta_1=\dots=\theta_{m-1}<\theta_m$ and $\lambda_0\theta_m\ge
\lambda_0+\theta_m$ (by Proposition \ref{prop:2-strange-indices}).
One obtains either $0\in\conv\{a,b\}$ or $e_m\in\pos\{a,b\}$ since
\[
	a+(\lambda_0-1)b=(\lambda_0\sum_{p\in S_a} p) - \lambda_0 e_m
	+ (\lambda_0-1)\theta_m e_m = ((\lambda_0-1)\theta_m -
	\lambda_0) e_m,
\]
where $(\lambda_0-1)\theta_m - \lambda_0\ge 0$ as $\lambda_0
\theta_m\ge \lambda_0+\theta_m$ (we have $0\in\conv\{a,b\}$ for the
equality).

Suppose $|S_a\cap S_b|=1$ with $|S_a|\le |S_b|$, say
$S_a=\{e_0,\dots,e_k\}, S_b=\{e_k,\dots,e_m\}, a=\sum_{t=0}^k\lambda_t
e_t, b=\sum_{t=k}^m\theta_t e_t$ with
$\theta_k\ge\theta_{k+1}=\dots=\theta_m$ (by Proposition
\ref{prop:support-insertecting-one-point->strange-index}).

If $1=\lambda_0=\dots=\lambda_{k-1}\le\lambda_k$, then
$e_k\in\pos\{a,b\}$ since
\[
	a+b=(\lambda_k+\theta_k)e_k.
\]

If $k=1$ and $\lambda_0>\lambda_1=1$, then $e_1\in\pos\{a,b\}$ since
\[
	\frac{1}{\lambda_0}a + b=(\frac{1}{\lambda_0}+\theta_1)e_1.
\]

In either of the two situations, we always have a contradiction to the
primitivity of $Y$.

All the cases were covered and we have finished the proof.

\section{Proofs of Propositions}
\label{sec:proofs-of-propositions}
This section presents the proofs of the propositions in Section
\ref{sec:reduction} and Section \ref{sec:k=1}.

\begin{proof}[Proof of Proposition \ref{prop:in-pos-of-2}]
	As $x$ is not in the space spanned by any $X_i$, there is a
	partition of $1,\dots,k$ into two nonempty parts, say $L\sqcup
	M=\{1,\dots,k\}$ such that $x$ is neither in the linear space
	$\mathcal L$ spanned by $X_L=\bigcup_{i\in L} X_i$ nor the
	linear space $\mathcal M$ spanned by $X_M=\bigcup_{i\in M}
	X_i$.

	Let the projection of $x$ onto the spaces $\mathcal L$ and
	$\mathcal M$ be $x_{\mathcal L}$ and $x_{\mathcal M}$,
	respectively.

	We show that $x_{\mathcal L}$ is in $X_L$.  Suppose otherwise,
	choose $\dim \mathcal L$ linearly independent points
	$\{y_i\}_i$ of $X_L$ whose positive hull contains $x_{\mathcal
	L}$.  Let $x_{\mathcal L} =\sum_i \lambda_i y_i$, then at
	least two $\lambda_i$ are positive.  Choose $\dim \mathcal M$
	linearly independent points $\{z_j\}_j$ of $X_M$ but their
	positive hull does not contain $x_{\mathcal M}$.  If
	$x_{\mathcal M}=\sum_j \theta_j z_j$, then at least one
	$\theta_j$ is negative.

	In total, they are $n$ linearly independent points, and $x$
	can be uniquely represented as
	\[
		x=x_{\mathcal L}+x_{\mathcal M}=\sum_i \lambda_i y_i +
		\sum_j \theta_j z_j.
	\]

	Since at least two $\lambda_i$ are positive and at least one
	$\theta_j$ is negative, we obtain a bad position of $x$ and
	the $n$ points, a contradiction.  (The readers may want to
	check Section \ref{sec:main-technique}.  The proofs of the
	other propositions also extensively use this method.)

	Therefore, $x_{\mathcal L}$ is just a positive multiple of a
	point in $X_L$.  Exchanging the roles of $\mathcal L$ and
	$\mathcal M$, we obtain the conclusion that $x$ is in the
	positive hull of some two points $x_i\in X_i, x_j\in X_j$ for
	different $X_i, X_j$.
\end{proof}

\begin{proof}[Proof of Proposition \ref{prop:disjoint-or-equal}]
	Suppose there exist $x\in\pos\{x_1,x_2\}$ for $x_1\in X_1,
	x_2\in X_2$ and $y\in\pos\{y_2,y_3\}$ for $y_2\in X_2, y_3\in
	X_3$.

	Let $x=\alpha_1 x_1 + \alpha_2 x_2$ and $y=\beta_2 y_2 +
	\beta_3 y_3$.

	Consider the case $x_2\ne y_2$, the linear relation
	\[
		\frac{1}{\alpha_2} x + \frac{1}{\beta_2} y =
		\frac{\alpha_1}{\alpha_2} x_1 +
		\frac{\beta_3}{\beta_2} y_3 - \sum_{p\in
		X_2\setminus\{x_2,y_2\}} p
	\]
	shows that the $n+1$ points in $\{x,y\}\cup X'_1\cup
	X_2\setminus\{x_2,y_2\}\cup X'_3\cup\dots\cup X'_k$ for any
	$X'_i\subset X_i$ satisfying $|X'_i|=|X_i|-1$, $x_1\in X'_1$,
	$y_3\in X'_3$ are in a bad position.

	Consider the case $x_2=y_2$, the linear relation
	\[
		\frac{1}{\alpha_2} x - \frac{1}{\beta_2} y =
		\frac{\alpha_1}{\alpha_2} x_1 -
		\frac{\beta_3}{\beta_2} y_3
	\]
	shows that the $n+1$ points in $\{x,y\}\cup X'_1\cup X''_2\cup
	X'_3\cup\dots\cup X'_k$ for any $X'_i\subset X_i$ satisfying
	$|X'_i|=|X_i|-1$, $x_1\in X'_1$, $y_3\in X'_3$, $X''_2\subset
	X_2$, $|X''_2|=|X_2|-2$, $x_2\notin X''_2$ are in a bad
	position.

	The conclusions follows from the bad positions in the cases
	covered.
\end{proof}

\begin{proof}[Proof of Proposition \ref{prop:same-region-same-point}]
	This proof is quite similar to the proof of Proposition
	\ref{prop:disjoint-or-equal}.

	Suppose $i=1,j=2$, let $x=\alpha_1 x_1 + \alpha_2 x_2$,
	$y=\beta_1 y_1 + \beta_2 y_2$.

	Assume the conclusion does not hold, that is $x_1\ne y_1$
	while $|X_1|>2$.

	Consider the case $x_2\ne y_2$, the linear relation
	\[
		\frac{1}{\alpha_2} x + \frac{1}{\beta_2} y =
		\frac{\alpha_1}{\alpha_2} x_1 +
		\frac{\beta_1}{\beta_2} y_1 - \sum_{p\in
		X_2\setminus\{x_2,y_2\}} p
	\]
	shows that the $n+1$ points in $\{x,y\}\cup X'_1\cup
	X_2\setminus\{x_2,y_2\}\cup X'_3\cup\dots\cup X'_k$ for any
	$X'_i\subset X_i$ satisfying $|X'_i|=|X_i|-1$, $x_1,y_1\in
	X'_1$ are in a bad position.

	Consider the case $x_2=y_2$, the linear relation
	\[
		\frac{1}{\alpha_2} x - \frac{1}{\beta_2} y =
		\frac{\alpha_1}{\alpha_2} x_1 -
		\frac{\beta_1}{\beta_2} y_1
	\]
	shows that the $n+1$ points in $\{x,y\}\cup X'_1\cup X''_2\cup
	X'_3\cup\dots\cup X'_k$ for any $X'_i\subset X_i$ satisfying
	$|X'_i|=|X_i|-1$, $x_1,y_1\in X'_1$, $X''_2\subset X_2$,
	$|X''_2|=|X_2|-2$, $x_2\notin X''_2$ are in a bad position.

	The bad positions imply that $x_1=y_1$.
\end{proof}

\begin{proof}[Proof of Proposition \ref{prop:restrict-support}]
	Consider a point $y\notin X_i$ in the space spanned by $X_i$.
	Let $y=\sum_{e_t\in S_y}\lambda_t e_t$.

	Assume $x_i\in S_y$, we have
	\[
		x_i=\frac{1}{\lambda_i} y - \sum_{e_t\in
		S_y\setminus\{x_i\}} \frac{\lambda_t}{\lambda_i} e_t.
	\]

	Choose the $k$ subsets $X'_t\subset X_t$ for $t=1,\dots,k$ so
	that $|X'_t|=|X_t|-1$, $S_y\subseteq X'_i$ and $x_j\in X'_j$.
	Since the points in $X'_i\cup \{y\}\setminus\{x_i\}$ are
	linearly independent, the following relation
	\[
		x= \alpha(\frac{1}{\lambda_i} y - \sum_{e_t\in
		S_y\setminus\{x_i\}} \frac{\lambda_t}{\lambda_i} e_t)
		+ \beta x_j
	\]
	for $x=\alpha x_i + \beta x_j$ shows a bad position for the
	$n+1$ points $\{x,y\}\cup (X'_1\cup\dots\cup
	X'_k)\setminus\{x_i\}$, contradiction.

	Therefore, $x_i\notin S_y$.
\end{proof}

\begin{proof}[Proof of Proposition \ref{prop:all-equal-n-positive}]
	Let $x=\alpha x_i + \beta x_j$.

	Assume there exists a point $y$ in the span of $X_i$ but not
	being a positive multiple of $\sum_{e_t\in S_y} e_t$.  Let the
	support of $y$ be $S_y=\{e_0,\dots,e_m\}\subset X_i$, by
	Proposition \ref{prop:almost-homo}, with the strange index
	supposed to be $e_0$, we have $y=\theta_0 e_0 + \sum_{t=1}^m
	e_t$ for $\theta_0>1$.

	By Proposition \ref{prop:restrict-support}, $x_i\notin S_y$.
	The linear relation
	\begin{align*}
		y&=(\theta_0-1)e_0 + \sum_{t=0}^m e_t\\
		&=(\theta_0-1)e_0 - \sum_{e_t\in X_i\setminus S_y} e_t
		\\
		&=(\theta_0-1)e_0 - \Bigl(\sum_{e_t\in (X_i\setminus
		S_y)\setminus \{x_i\}} e_t\Bigr) - \frac{1}{\alpha} x +
		\frac{\beta}{\alpha} x_j
	\end{align*}
	shows a bad position for the points $\{x,y\}\cup X'_1\cup
	\dots\cup X'_{i-1}\cup X_i\setminus\{x_i,x_{i'}\}\cup
	X'_{i+1}\cup \dots\cup X'_k$, where $X'_t\subset X_t$ is any
	set with $|X'_t|=|X_t|-1$ and $x_{i'}\in S_y\setminus
	\{e_0\}$, $x_j\in X'_j$.

	The conclusion on the relation between $S_p$ and $S_q$ can be
	deduced from Proposition \ref{prop:intersection} with the
	restriction from Proposition \ref{prop:restrict-support}.  We
	however give a self-contained argument below.

	Assume there exist $p, q$ in the space spanned by $X_i$ such
	that $S_p\cap S_q\ne\emptyset$, $S_p\setminus
	S_q\ne\emptyset$, $S_q\setminus S_p\ne\emptyset$.  Since
	$p=\sum_{e\in S_p} e$, $q=\sum_{e\in S_q} e$, the linear
	relation
	\[
		p-q=\sum_{e\in S_p\setminus S_q} e - \sum_{e\in
		S_q\setminus S_p} e
	\]
	shows a bad position for the $n+1$ points $\{p,q\}\cup
	X'_1\dots \cup X'_{i-1}\cup X_i\setminus\{x_i,x_{i'}\} \cup
	X'_{i+1} \cup \dots \cup X'_k$ for any $x_{i'}\in S_p\cap S_q$
	and any $X'_t\subset X_t$ with $|X'_t|=|X_t|-1$.
\end{proof}

\begin{proof}[Proof of Proposition \ref{prop:almost-homo}]
Assume otherwise, there are three indices of $S_p$, namely $0,1,2$,
such that $1=\lambda_0$, $1 < \lambda_1$, $1 < \lambda_2$ for
$p=\sum_{i=0}^k \lambda_i e_i$ (that is $S_p=\{e_0,\dots,e_k\}$).

	Since $e_0=-\sum_{i=1}^n e_i$, the unique relation
	\[
		p=\sum_{i=0}^k \lambda_i e_i = \sum_{i=1}^k \lambda_i
		e_i - \sum_{i=1}^n e_i = (\lambda_1 - 1)e_1 +
		(\lambda_2 - 1) e_2 + \sum_{i=3}^k (\lambda_i-1) e_i -
		\sum_{i=k+1}^n e_i
	\]
	shows that the $n+1$ points $p,e_1,\dots,e_n$ are in a bad
	position, contradiction.
\end{proof}

\begin{proof}[Proof of Proposition \ref{prop:intersection}]
	Suppose there are $p,q$ with $S_p\cap S_q\ne\emptyset,
	S_p\setminus S_q\ne\emptyset, S_q\setminus S_p\ne\emptyset$.
	Let $p=\sum_{e_i\in S_p} \lambda_i e_i$ and $q=\sum_{e_i\in
	S_q} \theta_i e_i$.

	We first show that $S_p\cup S_q= E$. Assume otherwise,
	consider any $e_j\in S_p\cap S_q$, we have
	\[
		\theta_j p - \lambda_j q = \sum_{e_i\in S_p\setminus
		S_q} \theta_j \lambda_i e_i - \sum_{e_i\in
		S_q\setminus S_p} \lambda_j \theta_i e_i +
		\sum_{e_i\in (S_p\cap S_q)\setminus\{e_j\}} (\theta_j
		\lambda_i - \lambda_j \theta_i)e_i.
	\]
	Since $S_p\cup S_q\ne E$, the point $p$ is not in the linear
	span of $(S_p\cup S_q)\setminus\{e_j\}$, by Lemma
	\ref{lem:base-for-linear-independence}.  The above linear
	relation shows that the $n+1$ points $\{p,q\}\cup (S_p\cup
	S_q)\setminus\{e_j\}$ are in a bad position, contradiction.

	Assume $|S_p\cap S_q|\ge 2$, we show that $|S_p\cap S_q|=n-1$.

	Take any $a,b$ such that $\lambda_a=\max_{i} \lambda_i$ and
	$\theta_b=\max_{i} \theta_i$.

	If $e_a, e_b \in S_p\cap S_q$, then the linear relation
	\[
		p+q=p+q-\sum_{e_i\in E} e_i = \sum_{e_i\in S_p\cap
		S_q} (\lambda_i + \theta_i - 1) e_i
	\]
	shows that the $n+1$ points $\{p,q\}\cup
	E\setminus\{e_i,e_j\}$ for any $e_i\in S_p\setminus S_q$ and
	any $e_j\in S_q\setminus S_p$ are in a bad position,
	contradiction.

	If $e_a\in S_p\cap S_q, e_b\in S_q\setminus S_p$, then the
	relation
	\begin{align*}
		\theta_b p + q&=\theta_b p + q - \theta_b\sum_{e_i\in
		E} e_i \\
		&= \sum_{e_i\in S_p\setminus S_q} \theta_b e_i +
		\sum_{e_i\in S_p\cap S_q} (\theta_b \lambda_i + 1) e_i
		+ \Bigl(\sum_{e_i\in (S_q\setminus
		S_p)\setminus\{e_b\}} e_i\Bigr) + \theta_b e_b -
		\theta_b\sum_{e_i\in E} e_i\\
		&= \sum_{e_i\in S_p\cap S_q} (\theta_b \lambda_i + 1 -
		\theta_b) e_i + \sum_{e_i\in (S_q\setminus
		S_p)\setminus\{e_b\}} (1-\theta_b )e_i
	\end{align*}
	shows that the $n+1$ points $\{p,q\}\cup
	E\setminus\{e_i,e_b\}$ for some $e_i\in S_p\setminus S_q$ are
	in a bad position, contradiction.

	If $e_a\in S_p\setminus S_q, e_b\in S_p\cap S_q$, we have a
	similar argument due to symmetry.

	If $e_a\in S_p\setminus S_q, e_b\in S_q\setminus S_p$, and
	$|S_p\setminus S_q|\ge 2$, then the relation
	\begin{align*}
		\theta_b p + q &= \theta_b \lambda_a e_a +
		\sum_{e_i\in (S_p\setminus S_q)\setminus \{e_a\}}
		\theta_b e_i + \sum_{e_i\in S_p\cap S_q} (\theta_b +
		1) e_i \\
		& \qquad + \Bigl(\sum_{e_i\in (S_q\setminus
		S_p)\setminus\{e_b\}} e_i\Bigr) + \theta_b e_b -
		\theta_b \sum_{e_i\in E} e_i\\
		&= (\theta_b\lambda_a - \theta_b) e_a + \sum_{e_i\in
		S_p\cap S_q} e_i + \sum_{e_i\in (S_q\setminus
		S_p)\setminus\{e_b\}} (1-\theta_b) e_i
	\end{align*}
	shows that the $n+1$ points $\{p,q\}\cup
	E\setminus\{e_i,e_b\}$ for some $e_i\in (S_p\setminus
	S_q)\setminus\{e_a\}$ are in a bad position, contradiction.

	A similar argument applies for the case $e_a\in S_p\setminus
	S_q, e_b\in S_q\setminus S_p$, and $|S_q\setminus S_p|\ge 2$.

	The only remaining possible case is $|S_p\setminus
	S_q|=|S_q\setminus S_p|=1$, that is $|S_p\cap S_q|=n-1$.  The
	conclusion follows.
\end{proof}
\begin{remark} \label{rem:strange-indices-in-iv}
	In Case (iv), if $p,q$ have strange indices, they can be the
	points in $S_p\setminus S_q$ and $S_q\setminus S_p$ only.
	This fact will be used in the proof of Proposition
	\ref{prop:2-strange-indices}.
\end{remark}

\begin{proof}[Proof of Proposition
\ref{prop:same-support-same-strange-index}]
	Assume distinct $e_a,e_b$ are the strange indices of $p,q$
	respectively.  Let $p=\sum_{e_i\in S_p} \lambda_i e_i$ and
	$q=\sum_{e_i\in S_q} \theta_i e_i$.  The linear relation
	\[
		p-q=(\lambda_a-1) e_a - (\theta_b-1) e_b
	\]
	shows that the points in $\{p,q\}\cup E\setminus\{e_i,e_j\}$
	for any $e_i\in S_p\setminus\{e_a,e_b\}$, $e_j\in E\setminus
	S_p$ are in a bad position, contradiction.
\end{proof}

\begin{proof}[Proof of Proposition
\ref{prop:disjoint-supports->at-most-1-has-strange-index}]
	Assume $e_a,e_b$ are the strange indices of $p,q$
	respectively.  Let $p=\sum_{e_i\in S_p} \lambda_i e_i$ and
	$q=\sum_{e_i\in S_q} \theta_i e_i$.  The linear relation
	\[
		p+q=p+q-\sum_{e_i\in E} e_i=(\lambda_a-1)e_a +
		(\theta_b-1)e_b - \sum_{e_i\in E\setminus(S_p\cup
		S_q)} e_i
	\]
	shows that the $n+1$ points in $\{p,q\}\cup
	E\setminus\{e_i,e_j\}$ for any $e_i\in S_p\setminus\{e_a\}$,
	$e_j\in S_q\setminus\{e_b\}$ are in a bad position,
	contradiction.
\end{proof}

\begin{proof}[Proof of Proposition \ref{prop:same-for-father-n-child}]
	Suppose $S_p=\{e_0,\dots,e_k\}$, $S_q=\{e_0,\dots,e_\ell\}$
	with $0<k<\ell<n$.

	Since there is always some $\lambda_i=1$, we assume that
	$\lambda_0=1$.  Consider the relation
	\[
		\lambda_0 q - \theta_0 p = \sum_{i=1}^k
		(\lambda_0\theta_i - \theta_0\lambda_i) e_i +
		\sum_{i=k+1}^\ell \lambda_0\theta_i e_i.
	\]

	We need $\lambda_0\theta_i - \theta_0\lambda_i\ge 0$ for each
	$i=1,\dots,k$, otherwise the $n+1$ points in $\{p,q\}\cup
	E\setminus\{e_0,e_n\}$ would be in a bad position.

	If $\theta_i=1$ for some $1\le i\le k$, then
	$\lambda_0\theta_i-\theta_0\lambda_i\ge 0$ implies
	$\theta_0=1$ and $\lambda_i=1$.

	If $\theta_i>1$ for some $1\le i\le k$, then
	$\lambda_0\theta_i-\theta_0\lambda_i\ge 0$ implies
	$\theta_i\ge\lambda_i$.  (Note that we still have $\theta_0=1$
	in this case as $\theta_i>1$.)

	We have shown that $\theta_i\ge\lambda_i$ for every $e_i\in
	S_p$.  Since there is always some $\theta_i=1$ for $e_i\in
	S_p$, we suppose it is the case for $i=0$, that is
	$\theta_0=\lambda_0=1$.  (Note that we do not need to assume
	$\lambda_0=1$ this time as $\lambda_0\le\theta_0$.)

	Suppose $\theta_i>\lambda_i$ for some $e_i$, say
	$\theta_k>\lambda_k$, we have
	\[
		\lambda_k q - \theta_k p = \sum_{i=0}^{k-1} (\lambda_k
		\theta_i - \theta_k \lambda_i) e_i + \sum_{i=k+1}^\ell
		\lambda_k\theta_k e_i.
	\]

	Since $\theta_0=\lambda_0=1$, it follows that
	$\lambda_k\theta_0 - \theta_k\lambda_0 < 0$.  It means the
	$n+1$ points in $\{p,q\}\cup E\setminus\{e_k,e_n\}$ are in a
	bad position, contradiction.

	Therefore, we obtain the conclusion that $\lambda_i=\theta_i$
	for $e_i\in S_p$.
\end{proof}

\begin{proof}[Proof of Proposition
\ref{prop:support-insertecting-one-point->strange-index}]
	Denote $p=\sum_{e_i\in S_p} \lambda_i e_i$ and $q=\sum_{e_i\in
	S_q} \theta_i e_i$.  Let $e_a\in S_p$, $e_b\in S_q$ so that
	$\lambda_a=\max_{e_i\in S_p} \lambda_i$ and
	$\theta_b=\max_{e_i\in S_q} \theta_i$.  If
	$\lambda_a=\lambda_k$, we set $a=k$.  Likewise, if
	$\theta_b=\theta_k$, we set $b=k$.

	Consider the case $|S_p|\ge 3$, $|S_q|\ge 3$.

	If $a\ne k$ and $b\ne k$, then the relation
	\[
		p+q=p+q-\sum_{e_i\in E} e_i=(\lambda_a-1)e_a +
		(\theta_b - 1)e_b + (\lambda_k+\theta_k-1) e_k
	\]
	shows that the $n+1$ points $\{p,q\}\cup
	E\setminus\{e_i,e_j\}$ for any $e_i\in
	S_p\setminus\{e_a,e_k\}$ and $e_j\in S_q\setminus\{e_b, e_k\}$
	are in a bad position, contradiction.

	If $a=k$ and $b\ne k$, then the relation
	\[
		p+q=p+q-\sum_{e_i\in E} e_i=(\lambda_k+\theta_k-1) e_k
		+ (\theta_b-1)e_b
	\]
	shows that the $n+1$ points $\{p,q\}\cup
	E\setminus\{e_i,e_j\}$ for any $e_i\in S_p\setminus\{e_k\}$
	and $e_j\in S_q\setminus \{e_k,e_b\}$ are in a bad position,
	contradiction.  (Note that this also works for $|S_p|=2$.)

	A similar argument works for the case $a\ne k$ and $b=k$.

	For the case $|S_p|=2$ (hence $|S_q|\ge 3$), we suppose
	$S_p=\{e_0,e_1\}$, $S_q=\{e_1,\dots, e_n\}$ (then $k=1$).

	If $a=1$, then the above linear relation, which also works for
	$|S_p|=2$, shows that $b=1$, otherwise we have same set of
	points in a bad position.

	If $a=0$, suppose $b\ne 1$, then the linear relation
	\[
		\frac{1}{\lambda_0} p + q = \frac{1}{\lambda_0} p + q
		- \sum_{e_i\in E} e_i = \frac{\lambda_1}{\lambda_0}
		e_1 + (\theta_b - 1)e_b
	\]
	shows that the $n+1$ points in $\{p,q\}\cup
	E\setminus\{e_0,e_j\}$ for any $e_j\in
	S_q\setminus\{e_1,e_b\}$ are in a bad position, contradiction.

	All the cases were covered, the conclusion follows.
\end{proof}

\begin{proof}[Proof of Proposition \ref{prop:2-strange-indices}]
	As in Remark \ref{rem:strange-indices-in-iv} after the proof
	of Proposition \ref{prop:intersection}, we have shown that for
	the condition given in Case (iv) of Proposition
	\ref{prop:2-strange-indices}, the strange indices of $p,q$, if
	any, are $e_a,e_b$, respectively.  It means
	$\lambda_i=\theta_i=1$ for $e_i\in S_p\cap S_q$.

	We now prove that $\lambda_a>1$, $\theta_b>1$ and $\lambda_a
	\theta_b \ge \lambda_a + \theta_b$.  Assume otherwise, the
	linear relation
	\begin{align*}
		\theta_b p + \lambda_a q &= \theta_b \lambda_a e_a +
		\lambda_a \theta_b e_b + \sum_{e_i\in S_p\cap S_q}
		(\lambda_a+\theta_b) e_i - \lambda_a \theta_b
		\sum_{e_i\in E} e_i\\
		&= \sum_{e_i\in S_p\cap S_q} (\lambda_a+\theta_b -
		\lambda_a\theta_b) e_i
	\end{align*}
	has the positive coefficient $\lambda_a+\theta_b -
	\lambda_a\theta_b$ for $e_i\in S_p\cap S_q$ (note that
	$|S_p\cap S_q|\ge 2$), a contradiction by the bad position of
	the points in $\{p,q\}\cup E\setminus\{e_a,e_b\}$.
\end{proof}

\begin{proof}[Proof of Proposition
\ref{prop:two-big-supports-intersecting-at-one-point}]
	Assume otherwise, there exist $e_i\in S_r\setminus S_q$ and
	$e_j\in S_r\setminus S_p$.  In fact, $e_i\in S_p\setminus S_q$
	and $e_j\in S_q\setminus S_p$.

	Since $S_r\ne E$, either $S_p$ or $S_q$ is not a subset of
	$S_r$.  Suppose $S_q$ is not a subset of $S_r$.  It follows
	that $|S_r\cap S_q|$ is either $1$ or $n-1$ (by Proposition
	\ref{prop:intersection}).

	Suppose $|S_r\cap S_q|=1$, that is $S_r\cap S_q=\{e_j\}$.  It
	follows that $S_r=E\setminus S_q\cup\{e_j\}$ by Proposition
	\ref{prop:intersection}.  However, we then have $|S_p\cap
	S_r|=|S_p|-1$, which is neither $1$ nor $n-1$, contradicting
	Proposition \ref{prop:intersection}.

	It follows that $|S_r\cap S_q|=n-1$.  However, $|S_r\cap S_q|
	< |S_q| \le n-1$ since $|S_p|\ge 3$, contradiction.
\end{proof}

\begin{proof}[Proof of Proposition
\ref{prop:big-intersection-constrain-others}]
	The proposition is obviously true for $n=2$.  We assume $n\ge
	3$.

	Suppose $S_p=\{e_0,\dots,e_{n-1}\}$, $S_q=\{e_1,\dots,e_n\}$.
	Assume $S_r$ is not a subset of $S_p\cap S_q$, we show that
	$S_r=\{e_0,e_n\}$.

	Let $p=\sum_{e_i\in S_p} \lambda_i e_i$ and $q=\sum_{e_i\in
	S_q} \theta_i e_i$.

	Firstly, $|S_r|<n$.  Assume otherwise, let
	$S_r=E\setminus\{e_k\}$ ($k\ne 0,n$).  Because $S_r\cup S_p=E,
	|S_r\cap S_p|=n-1$ and $S_p\setminus S_r=\{e_k\}$, one obtains
	$\lambda_k>1$ but we already have $\lambda_0>1$ (by
	Proposition \ref{prop:2-strange-indices}), contradiction.

	We continue by showing that $|S_r|$ cannot exceed $2$.  Assume
	otherwise that $|S_r|>2$.  As $S_r$ is not a subset of
	$S_p\cap S_q$, either $e_0$ or $e_n$ is in $S_r$, say $e_0\in
	S_r$.  Hence, $|S_r\cap S_q| = |S_r|-1$.  It follows that
	$1<|S_r\cap S_q|<n-1$, contradicting Proposition
	\ref{prop:intersection}.

	As $|S_r|=2$, if $S_r$ is not a subset of $S_p\cap S_q$, then
	either $S_r=\{e_0,e_k\}$, or $S_r=\{e_k,e_n\}$, or
	$S_r=\{e_0,e_n\}$ for some $k\ne 0,n$.  Assume
	$S_r=\{e_0,e_k\}$, then $S_r\cap S_q=\{e_k\}$, by Proposition
	\ref{prop:support-insertecting-one-point->strange-index} we
	obtain $\theta_k>1$ as the strange index of $q$, but we
	already have $\theta_n>1$, a contradiction.  A similar
	contradiction is raised when $S_r=\{e_k,e_n\}$.  Therefore,
	$S_r=\{e_0,e_n\}$.
\end{proof}

\begin{proof}[Proof of Proposition \ref{prop:small-intersect-big}]
	The proposition is obviously true for $n=2$.  We assume $n\ge
	3$.

	Suppose $S_p=\{e_0,e_1\}$, $S_q=\{e_1,\dots,e_n\}$.  Assume
	$S_r$ is not a subset of $S_q\setminus S_p$, that is $S_r$
	contains either $e_0$ or $e_1$.

	If $|S_r|=2$ then $S_r$ is either $\{e_0,e_k\}$ or
	$\{e_1,e_k\}$ for $k\ne 0,1$.  The intersection $S_r\cap S_p$
	is only a point but their union is not $E$, contradicting
	Proposition \ref{prop:intersection}.

	Assume $2<|S_r|<n$, one obtains $e_0\notin S_r$, as otherwise
	$1<|S_r\cap S_q|=|S_r|-1<n-1$.  Also, $e_1\notin S_r$, as
	otherwise $S_p\cup S_r$ is not $E$, contradiction.  It follows
	that $|S_r|=n$.

	Suppose $S_r=E\setminus\{e_k\}$ for $k\ne 1$.  As $|S_r\cap
	S_q|=n-1$, it follows that $q$ has the strange index $e_k$ by
	Proposition \ref{prop:2-strange-indices}.  However, if $q$ has
	a strange index, the strange index must be $e_1$ by
	Proposition
	\ref{prop:support-insertecting-one-point->strange-index} since
	$|S_p\cap S_q|=1$, contradiction.  Therefore,
	$S_r=E\setminus\{e_1\}=(S_p\setminus S_q)\cup (S_q\setminus
	S_p)$.
\end{proof}

\section*{Acknowledgement}
The author would like to thank Roman Karasev for his encouragement in
characterizing this class of polytopes and his patient reading and
commenting various pieces in early drafts.  Also, the
author is grateful to Rolf Schneider, who has kindly given useful
comments improving the presentation of the introduction.

\bibliographystyle{unsrt}
\bibliography{linear}

\end{document}